\documentclass[a4paper]{amsart}

\usepackage{graphics,color,pgf}
\usepackage{epsfig}
\usepackage[ansinew]{inputenc}
\usepackage[all]{xy}
\usepackage{hyperref}
\usepackage{amssymb, amsmath,amsthm, mathtools, amscd}
\usepackage{mathrsfs}
\usepackage{stmaryrd}
\usepackage{cancel} %strike out
\usepackage{tikz}
\usepackage{tikz-cd}
\usetikzlibrary{matrix}
\usepackage{epigraph}
\usepackage{float}
\usepackage{subcaption}
\usepackage{graphics,color,pgf}
\usepackage{epsfig}
\usepackage[british,english]{babel}

\newtheorem{lem}{Lemma}[section]
\newtheorem{thm}[lem]{Theorem} 
\newtheorem{prop}[lem]{Proposition}

\theoremstyle{definition}
\newtheorem{dfn}[lem]{Definition}

\theoremstyle{remark}

\newtheoremstyle{TheoremNum}
        {0.2 cm}{0.2 cm}              %%% space between body and thm
        {\itshape}                      %%% Thm body font
        {}                              %%% Indent amount (empty = no indent)
        {}                     %%% Thm head font
        {.}                             %%% Punctuation after thm head
        { }                             %%% Space after thm head
        {\thmname{\bfseries #1}\thmnote{ \bfseries #3}}%%% Thm head spec
    \theoremstyle{TheoremNum}

\renewcommand{\leq}{\leqslant}
\renewcommand{\geq}{\geqslant}
\renewcommand{\setminus}{\smallsetminus}
\title[Exponential growth rates of even Coxeter groups]{Exponential growth rates of even Coxeter groups}

\author[V\'era Bossart]{%\bfseries 
V\'era Bossart}
\address{Universit\'e de Gen\`eve}
\email{vera.bossart@unige.ch}

\author[Michelle Bucher]{%\bfseries 
Michelle Bucher} 
\address{Universit\'e de Gen\`eve}
\email{michelle.bucher@unige.ch}

\thanks{The second author is supported by the Swiss National Science Foundation. \\
} %% optional
\begin{document}

\begin{abstract} Let $W$ be an even Coxeter group. We prove that among all Coxeter systems generating $W$ the unique even Coxeter system realizes the minimal exponential growth. Our proof relies on comparing the exponential growth rates in the explicit algorithm of Mihalik \cite{Mihalik} which from any Coxeter system of an even Coxeter group eventually produces the unique even one. The main new ingredient is that blow downs along pseudo-transpositions do not increase the exponential growth rate. 

 \end{abstract}

\maketitle

\section{Introduction}

Let $G$ be a finitely generated group. Recall that the \emph{exponential growth rate} of $G$ with respect to a finite generating set $S$ is defined as 
$$\omega(G,S)=\limsup_n \sqrt[n]{|\{ g \in G\mid \ell_S(g)=n \}|},$$
where $\ell_S:G\rightarrow \mathbb{N}$ is the length with respect to $S$,
$$\ell_S(g)=\min \{n \in \mathbb{N}\mid g=s_1\cdot \ldots \cdot s_n, \ s_i\in S\cup S^{-1}\}.$$
The exponential growth rate  clearly depends on the generating set $S$. In order to obtain a group invariant, it is natural to define the \emph{uniform exponential growth rate} as
$$\omega(G):=\inf_{S}\omega(G,S),$$
where the infimum is taken over all finite sets $S$ generating $G$. 

There are very few groups of exponential growth (i.e. such that $\omega(G,S)>1$ for one and hence for all generating sets $S$) for which one can compute  the explicit value of  $\omega(G)$. The very first example is by Gromov for the free group $F_k$ on $k$ generators, where one has $\omega(F_k)=2k-1$ \cite[Example 5.13]{Gromov}. Most famously, Wilson exhibited a family of groups of exponential growth $G$ for which $\omega(G)=1$ \cite{Wilson}, hereby giving a negative answer to a question of Gromov. Further examples where the value of $\omega(G)$ has been computed (and is $>1$) include particular free products, Baumslag-Solitar groups or lamplighter groups \cite{Bucher-Talambutsa-IJM,Bucher-Talambutsa-GGD,Mann,Talambutsa}. Astonishingly, for closed surfaces $\Sigma$ of higher genus, $\omega(\pi_1(\Sigma))$ is still unknown. 

In this paper, we will simplify the studied group invariant by restricting to Coxeter groups $W$ and considering 
$$\omega_{\mathrm{Cox}}(W):=\inf_{\mathrm{Cox}\ S}\omega(W,S),$$
where the infimum is now only taken over all finite Coxeter systems $S$ generating $W$. We will see examples of Coxeter groups admitting non conjugated Coxeter systems below.

Recall that a Coxeter group $W$ is said to be \emph{even} if there exists a Coxeter system $(W,S_\mathrm{even})$ such that all the labels of the corresponding Coxeter diagram are even or $\infty$. A remarkable result of Bahls \cite{Bahls} shows that if such an even Coxeter system $S_\mathrm{even}$ exists, it is unique. Our main result is:

\begin{thm}\label{thm even} Let $W$ be an even Coxeter group. Then
$$\omega_{\mathrm{Cox}}(W)=\omega(W,S_{\mathrm{even}}),$$
where $S_{\mathrm{even}}$ is the (unique) even Coxeter system for $W$.
\end{thm}

%Note that $S_{\mathrm{even}}$ happens to be the Coxeter system of minimal cardinality. 

Our proof relies on an algorithm of Mihalik \cite{Mihalik}, which starting from any non even Coxeter system $(W,S)$ either fails in which case the Coxeter group is not even, or produces a Coxeter system $(W,S')$ with $|S'|=|S|-1$. If the Coxeter group $W$ is even the algorithm will eventually produce the unique even Coxeter system. The procedure, which we will detail in Section \ref{algo} consists of a diagram twisting which has no effect on the growth rate, followed by so called \emph{blow downs along pseudo-transpositions}. Our main contribution is, very loosely formulated: 
\begin{thm} \label{thm main blow down} Blow downs along pseudo-transpositions do not increase the exponential growth rate. 
\end{thm}
The precise statement of  Theorem \ref{thm main blow down} is given in Theorem \ref{replacement decrease}. 

Comparison of exponential growth rates in a similar spirit were established by Terragni in \cite[Theorem A]{Terragni}, where the author shows that the growth rate is monotonous with respect to a natural partial order on Coxeter systems (on typically different Coxeter groups). This does however not apply to our setting as the two Coxeter systems we consider are unrelated by this order.

Among the (few) examples of groups for which the growth rate is known one finds the Coxeter groups $C_2 *(C_2\times C_2)$ and $\mathrm{PGL}(2,\mathbb{Z})$ \cite{Bucher-Talambutsa-IJM}, where $C_2$ denotes the cyclic group of order $2$. In those examples, it turns out that
$$\omega(G)=\omega_{\mathrm{Cox}}(W)$$ 
is realized on the generating (unique) Coxeter systems. The inequality $\omega_{\mathrm{Cox}}(W)\geq \omega(W)$ can however be strict as we will see on the example of the Coxeter group $C_2*\mathrm{Sym}(5)$ (see example \ref{Sym5}).

\subsection*{Acknowledgements:} We are grateful to Pierre de la Harpe for several comments that improved the exposition of this article, to Laurent Bartholdi and Corentin Bodart for many interesting discussions during the preparation of this paper, and  to Yves Cornulier for suggesting to look at (counter-) examples built on symmetric groups.

The results presented in this paper are parts of the Master thesis of the first author under the supervision of the second. 

\section{Background on Coxeter groups}

We will use standard notation as for example in  \cite{Humphreys} which we refer to for further details. 

A Coxeter system $(W,S)$ is defined by its Coxeter matrix $(m_S(s,t))_{s,t\in S}$, where $m_S(s,t)$ is the order of $st$ if $st$ is of finite order and $m_S(s,t)=\infty$ otherwise. In particular, $m_S$ is symmetric, $m_S(s,s)=1$ and $m_S(s,t)\geq 2$ whenever $s\neq t$. While $W$ can either be finite or infinite, we will only consider finite generating sets $S$. 

The Coxeter matrix can be encoded by a Coxeter diagram, which is a labelled graph with vertices $S$ and edges between $s\neq t\in S$ labelled by $m_S(s,t)$. When depicting our Coxeter diagrams we adopt the convention of omitting the edges labelled with a $2$. Note however that in the body of some proofs (for the Example \ref{Sym5} or in the proof of Theorem \ref{thm even} in Section \ref{algo}) it will be convenient to consider the labelled graph keeping the edges labelled with $2$ and forgetting the edges with $\infty$ labels. 

For any Coxeter system $(W,S)$, define $a_n$, for any $n\in \mathbb{N}$, to be the number of elements $w\in W$ of length $n$,  where the length $\ell_S(w)$ of $w\in W$ is as defined in the introduction. The Poincar\'e serie of $(W,S)$ is the formal serie
$$W(x)=\sum_{n=0}^{\infty}a_nx^n.$$
We will make crucial use of Steinberg's formula: 

\begin{thm}[Steinberg's Formula] \cite[Proposition 8]{Steinberg}\label{Steinberg}
Let $(W,S)$ be a Coxeter system. Then
$$\frac{1}{W_S(x^{-1})}=\sum_{U\subset \mathcal{S}^S_{\mathrm{fin}}} \frac{(-1)^{|U|}}{W_U(x)},$$
where the set $\mathcal{S}^S_{\mathrm{fin}}$ consists of all subsets $U\subset S$ such that the corresponding Coxeter group generated by $U$ is finite. 
\end{thm}

%\begin{thm}[Steinberg's Formula] Let $(W,S)$ be a Coxeter 
%Steinberg Formula.
%\end{thm}

\section{Pseudo-transposition}\label{section pseudo transposition}

Pseudo-transpositions are defined in \cite{Muhlherr}  as a mean to construct different Coxeter generating sets on the same Coxeter group $G$: 

\begin{dfn} \label{def pseudotransp} Let $(W,S)$ be a Coxeter system. An element $s_1 \in S$ is called \emph{pseudo-transposition} if the following conditions are satisfied:
\begin{enumerate}
\item There exists a unique $s_0\in S$ such that $m_S(s_0,s_1)=2(2k+1)$ for some $k\in \mathbb{N}^*$.
\item For every $s\in S\backslash \{s_0,s_1\}$ we have $m_S(s,s_1)\in \{ 2,\infty\}$. 
\item If for $s\in S\backslash\{s_0,s_1\}$ it holds that $m_S(s,s_1)=2$, then $m_S(s,s_0)=2$.
\end{enumerate}
\end{dfn}

\begin{thm} \cite[Lemma 3.4]{Muhlherr}\label{thm: T Coxeter set} Let $(W,S)$ be a Coxeter system. Suppose that $s_1\in S$ is a pseudo-transposition and let $s_0\in S$ be the unique element such that $m_S(s_0,s_1)=2(2k+1)$ for $k\geq 1$. Set
$$ s_2:=s_1s_0s_1\quad \mathrm{and} \quad r:=(s_1s_0)^{2k+1}.$$
Then the set 
$$T:= (S\backslash \{ s_1\})\cup \{s_2,r\}$$
is a Coxeter generating set for $W$. 
\end{thm}

This theorem is stated in  \cite[Lemma 3.4]{Muhlherr} without proof. A sketch is provided in \cite[Lemma 4.2]{Galaxy}. For the convenience of the reader we include a detailed demonstration. Before doing so, we illustrate the geometric meaning of the elements $s_2$ and $r$ on the motivating example of the dihedral group.

We denote by $D_{2n}$ the dihedral group of order $2n$ which we identify with the isometry group of a regular $n$-gon in the Euclidean plane centered at the origin. It is generated by two reflexions $s_0,s_1$ whose reflexion axis intersect with an angle of $\pi/n$. The set $S=\{s_0,s_1\}$ is a Coxeter generating set for $D_{2n}$ of Coxeter diagram
\begin{figure}[H]%                 use [hb] only if necceccary!
  \centering
  \includegraphics[width=2.5cm]{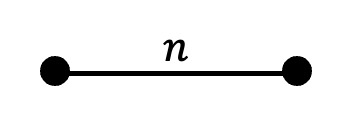}
 % \caption*{TestPicture}
%  \label{fig:test}
\end{figure}

In the case when $n$ is even then $D_{2n}$ contains the rotation by $\pi$ centered at the origin, which is given by $r=(s_1s_0)^{n/2}$. In view of its geometric description ($D_{2n}$ is a subgroup of $\mathrm{O}(2)$ and $r=-\mathrm{Id}$), it is clearly a central element of order $2$. If further $n=2(2k+1)$ for some $k\geq 1$, then $s_0$ and $s_2:=s_1s_0s_1$ generate a (proper) subgroup of $D_{4(2k+1)}$ isomorphic to $D_{2(2k+1)}$. Since $r$ does not belong to $D_{2(2k+1)}$ (because $2k+1$ is odd), we obtain an isomorphism
$$\langle s_0,s_1\rangle =D_{4(2k+1)}\cong D_{2(2k+1)}\times C_2=\langle s_0,s_2\rangle \times \langle r\rangle.$$
In particular, when $n=2(2k+1)$, the Coxeter diagram of $D_{4(2k+1)}$ with respect to the generating set $\{s_0,s_2,r\}$ is
\begin{figure}[H]%                 use [hb] only if necceccary!
  \centering
  \includegraphics[width=5cm]{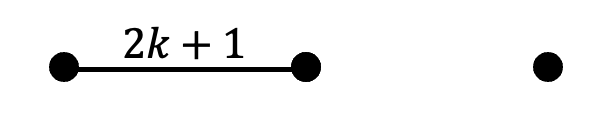}
 % \caption*{TestPicture}
%  \label{fig:test}
\end{figure}

In order to understand how to express $s_1$ in terms of $s_0,s_2$ and $r$, observe that for any linear rotation $R_\alpha$ of rotation angle $\alpha$ and any linear reflexion $s_\gamma$ through a line $\gamma$, the composition $R_\alpha(s_\gamma)$ is the reflexion through the line $R_{\alpha/2}(\gamma)$. After a choice of orientation, (one of ) the angles between the reflecting lines between $s_1$ and $s_0$ is $\pi/2(2k+1)$. The rotation $r(s_2s_0)^k$ is by an angle of $-\pi/(2k+1)$. In particular
$$s_1=r(s_2s_0)^ks_2.$$ 
If one prefers a completely algebraic approach one can also simply compute 
$$s_1=r^2 s_1 =r(s_1s_0)^{2k+1}s_1=r(\underbrace{s_1s_0s_1}_{=s_2}s_0)^{k}\underbrace{(s_1s_0)s_1}_{=s_2}= r(s_2s_0)^ks_2.$$

Note that when $n$ does not have the form $n=2(2k+1)$ such an isomorphism is impossible: If $n$ is not even this is because dihedral groups have even order, and if $n$ is a multiple of $4$ the groups $D_{2n}$ and $D_{n}\times C_2$ have non isomorphic abelianizations. It is easy to see what goes wrong in the above construction: If $n$ is not even then the group generated by $s_0$ and $s_2=s_1s_0s_1$ is not an index $2$ subgroup but the full group $D_{2n}$, and if $n$ is a multiple of $4$ then the rotation $r=(s_1s_0)^{n/2}$ belongs to the subgroup generated by $s_0$ and $s_2=s_1s_0s_1$.

\begin{proof}[Proof of Theorem \ref{thm: T Coxeter set}] It follows from the above example of the dihedral group $D_{2(2(2k+1))}$ that $T$ also is a generating set of $W$. Indeed we have replaced the generator $s_1$ by  $\{s_2,r\}$. But $\{s_0,s_2,r\}$ generate the dihedral group $\langle s_0,s_1\rangle$ and in particular contain $s_1$. %We can see this also purely algebraically:
%$$s_1=r^2 s_1 =r(s_1s_0)^{2k+1}s_1=r(\underbrace{s_1s_0s_1}_{=s_2}s_0)^{k}\underbrace{(s_1s_0)s_1}_{=s_2}= r(s_2s_0)^ks_2\in \langle T\rangle.$$

To see that $T$ is a Coxeter generating system, we define 
$$m_T(t,t'):= \mathrm{order \ of \ }tt' \mathrm{\ in \ } \langle T \rangle= W$$
for every $t,t'\in T$. Observe that $m_T(t,t)=1$ for every $t\in T$. This is clear for $t\in T \backslash \{s_2,r\}$. For $t=s_2$ or $r$ we verify that
\begin{eqnarray*}
s_2s_2&=&(s_1s_0s_1)(s_1s_0s_1)=e\\
rr&=&(s_1s_0)^{2k+1}(s_1s_0)^{2k+1}=(s_1s_0)^{2(2k+1)}=(s_1s_0)^{m_S(s_0,s_1)}=e.
\end{eqnarray*}
It follows that $m_T$ is a Coxeter matrix and we consider the corresponding abstract Coxeter group $W_T$ with generating set $T$ and Coxeter matrix $m_T$. By construction, we have a surjective homomorphism
$$\pi: W_T\rightarrow W=\langle T \rangle$$
sending $t\in W_T$ to $t\in W=\langle T\rangle$ for every $t\in T$. To show that $\pi$ is an isomorphism we define a homomorphism
$$\varphi:W=\langle S\rangle \longrightarrow W_T$$
by
\begin{eqnarray*}
\varphi(s)&=&s, \quad \mathrm{if \ }s\in S\backslash \{s_1\},\\
\varphi(s_1)&=&r(s_2s_0)^ks_2.
\end{eqnarray*}
To check that $\varphi$ is indeed a homomorphism we need to verify that $(\varphi(s)\varphi(s'))^{m_S(s,s')}=e$ for every $s,s'\in S$ such that $m_S(s,s')<\infty$.  
If $s,s'\in S\backslash \{s_1\}$ we have $\varphi(s)=s$, $\varphi(s')=s'$ and $m_S(s,s')=m_T(s,s')$ so that there is nothing to prove. Suppose that $s'=s_1$. 

For $s=s_0$ we have $m_S(s_0,s_1)=2(2k+1)$ and we compute
$$(\varphi(s_0)\varphi(s_1))^{2(2k+1)}=(s_0r(s_2s_0)^ks_2)^{2(2k+1)}=(r^2)^{2k+1}((s_0s_2)^{2k+1})^{2(k+1)}=e,$$
where we have used that $m_T(r,r)=1$ and $m_T(s_0,s_2)=2k+1$. 

For $s=s_1$ we have $m_S(s_1,s_1)=1$ and we compute
$$\varphi(s_1)\varphi(s_1)=r(s_2s_0)^ks_2r(s_2s_0)^ks_2=r^2s_2(s_0s_2)^k(s_2s_0)^ks_2=e.$$

For $s\neq s_0,s_1$ we either have $m_S(s,s_1)=\infty$, in which case there is nothing to prove, or $m_S(s,s_1)=2$. Note that in the latter case the definition of a pseudo-transposition imposes that $m_S(s,s_0)=2$, so that $s$ commutes with both $s_0$ and $s_1$ and thus also with $s_2$ and $r$.  In particular, $\varphi(s)$ commutes with $\varphi(s_1)$. 

Finally we show that $\varphi \circ \pi=\mathrm{Id}_{W_T}$: It is enough to check this on the set $T$. If $s\in T\setminus \{s_2,r\}$ this is clear by definition. For $s=s_2$ and $s=r$ we have
\begin{eqnarray*}
\varphi(\pi(s_2))&=&\varphi(s_2)=\varphi(s_1s_0s_1)=\varphi(s_1)\varphi(s_0)\varphi(s_1)=r(s_2s_0)^ks_2 s_0 r(s_2s_0)^ks_2\\
&=&r^2 (s_2s_0)^{2k+1}s_2=s_2,\\
\varphi(\pi(r))&=&\varphi(r)=\varphi((s_1s_0)^{2k+1})=(\varphi(s_1)\varphi(s_0))^{2k+1}=(r(s_2s_0)^ks_2 s_0)^{2k+1}\\
&=&r^{2k+1}((s_2s_0)^{2k+1})^k(s_2s_0)^{2k+1}=r.
\end{eqnarray*}
In particular, $\pi$ is injective. Since we already know that it is surjective, $\pi$ is indeed an isomorphism. 
\end{proof}

In the proof of the theorem, we have used the fact that whenever $m_S(s,s_1)=2$, for $s\in S\backslash \{s_0,s_1\}$ then $m_T(s,s_2)=m_T(s,r)=2$, which directly follows from the condition (3) of the definition of pseudo-transposition. According to the condition (2) of the definition, the only other possibility for $m_S(s,s_1)$ is $\infty$. In that case we can also conclude that $m_S(s,s_1)=m_T(s,s_2)=m_T(s,r)$ but this requires an  argument: 

\begin{prop}\label{prop labels} Let $s\in S\backslash \{s_0,s_1\}$. If $m_S(s,s_1)=\infty$, then $m_T(s,s_2)=m_T(s,r)=\infty$. 
\end{prop}

A purely algebraic proof of the proposition is sketched in  \cite[Lemma 4.2]{Galaxy}. We propose a geometric demonstration. 

\begin{proof}  Let $s\in S\backslash \{s_0,s_1\}$ and suppose that $m_S(s,s_1)=\infty$. We need to show that the orders of $ss_2$ and $sr$ are infinite. As $s$, $s_1$, $s_2=s_1s_0s_1$ and $r=(s_1s_0)^{2k+1}$ all lie in the triangular group $W_{\{s_0,s_1,s\}}$, we can without loss of generality suppose that $S=\{s_0,s_1,s\}$ and $W=W_{\{s_0,s_1,s\}}$. As we assume $m_S(s,s_1)=\infty$ the group $W$ is a hyperbolic triangular group which we can identify with the reflexion group of a hyperbolic triangle $\Delta$. We set ourselves in the upperhalf plane model for the hyperbolic plane $\mathbb{H}^2$ and choose $\Delta$ to have one vertex $\infty$, and the two others $x,y$ on the unit cercle $S^1$ centered at the origin, with $\mathrm{Re}(x)<\mathrm{Re}(y)$. In this way we can view the three reflecting geodesics $\gamma_s$, $\gamma_{s_0}$, $\gamma_{s_1}$ of $s,s_0$ and $s_1$ as follows:  $\gamma_s$ and $\gamma_{s_1}$ are the (vertical) geodesic through $x$, respectively $y$, and $\infty$, and $\gamma_{s_0}$ is the intersection of the unit cercle $S^1$ with the upperhalf plane. We set $\ell:=m_S(s,s_0)$. The point $x$ is thus chosen so that the angle between the faces $\gamma_s$ and $\gamma_{s_0}$ of $\Delta$ is $\pi/\ell$, while $y$ is chosen so that the angle between the faces $\gamma_1$ and $\gamma_0$ of $\Delta$ is $\pi/(2k+1)$. In this setting we can identify $s,s_0,s_1$ with the reflexions through the geodesics $\gamma_s,\gamma_{s_0},\gamma_{s_1}$ respectively and identify $W$ with the corresponding subgroup of $\mathrm{Isom}(\mathbb{H}^2)$. 
\begin{figure}[H]%                 use [hb] only if necceccary!
  \centering
  \includegraphics[width=7cm]{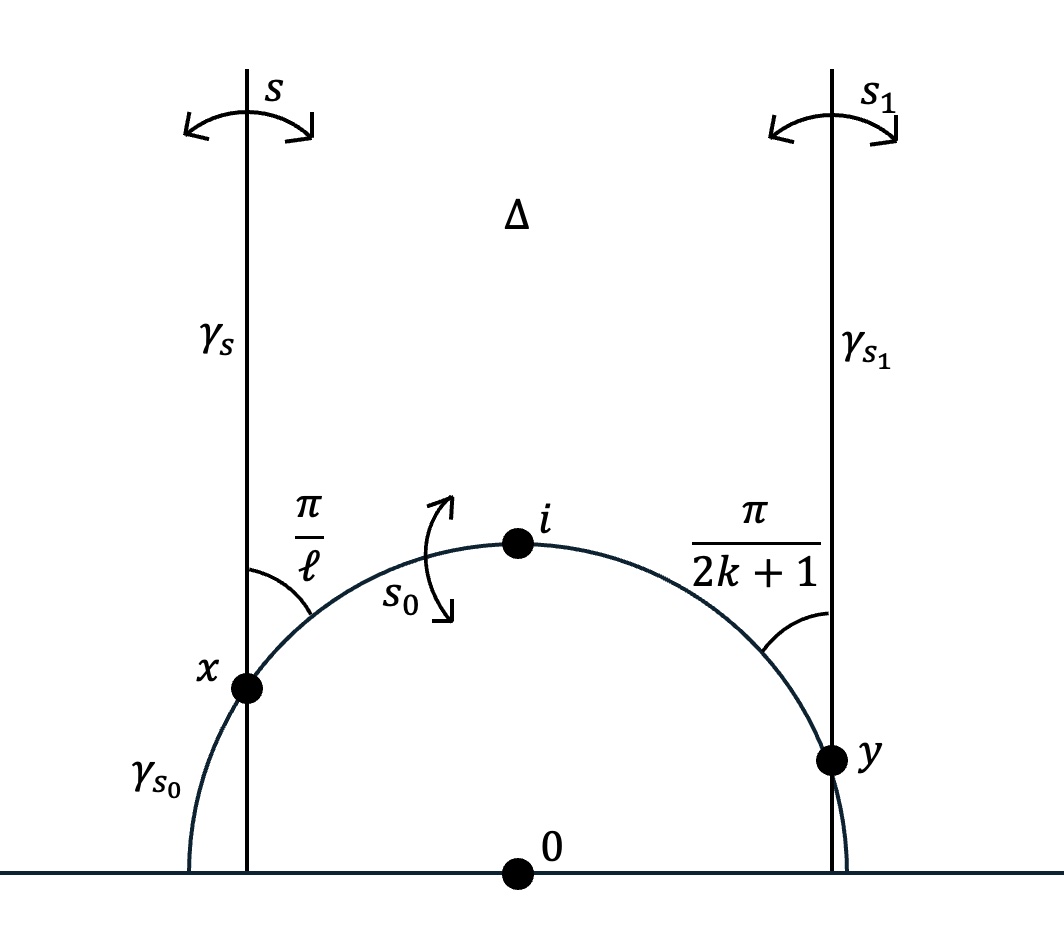}
 % \caption*{TestPicture}
%  \label{fig:test}
\end{figure}
Observe that since $\ell\geq 2$, the point $x$ satisfies $-1\leq \mathrm{Re}(x)\leq 0$ (with $-1=\mathrm{Re}(x)$ if and only if $\ell=\infty$ and $\mathrm{Re}(x)=0$ if and only if $\ell=2$). As the angle $\alpha$ of $\Delta$ at $y$ satisfies $\pi/3 \leq \alpha=\pi/(2k+1)<\infty$, we deduce that $1/2\leq \mathrm{Re}(y)<1$. 

In our setting, both $s_2$ and $r$ lie in the stabilizer of $y$. The reflexion $s_2=s_1s_0s_1$ is the reflexion by the geodesic $s_1(\gamma_0)$, while $r$ is the rotation by $\pi$ centered at $y$. 

Order of $sr$: We let $\gamma_r$ be the unique geodesic through $y$ perpendicular to $\gamma_s$. Both $s$ and $r$ send $\gamma_r$ to itself. The action of $s$ and $r$ restricted to $\gamma_r$ correspond to the reflexions by the points $\gamma_s\cap \gamma_r$ and $y$ respectively. Since these two points are distinct, $\langle s,r\rangle=\langle s\rangle *\langle r\rangle \cong D_\infty$ and the order of $sr$ is infinite as expected. 
\begin{figure}[H]%                 use [hb] only if necceccary!
  \centering
  \includegraphics[width=7cm]{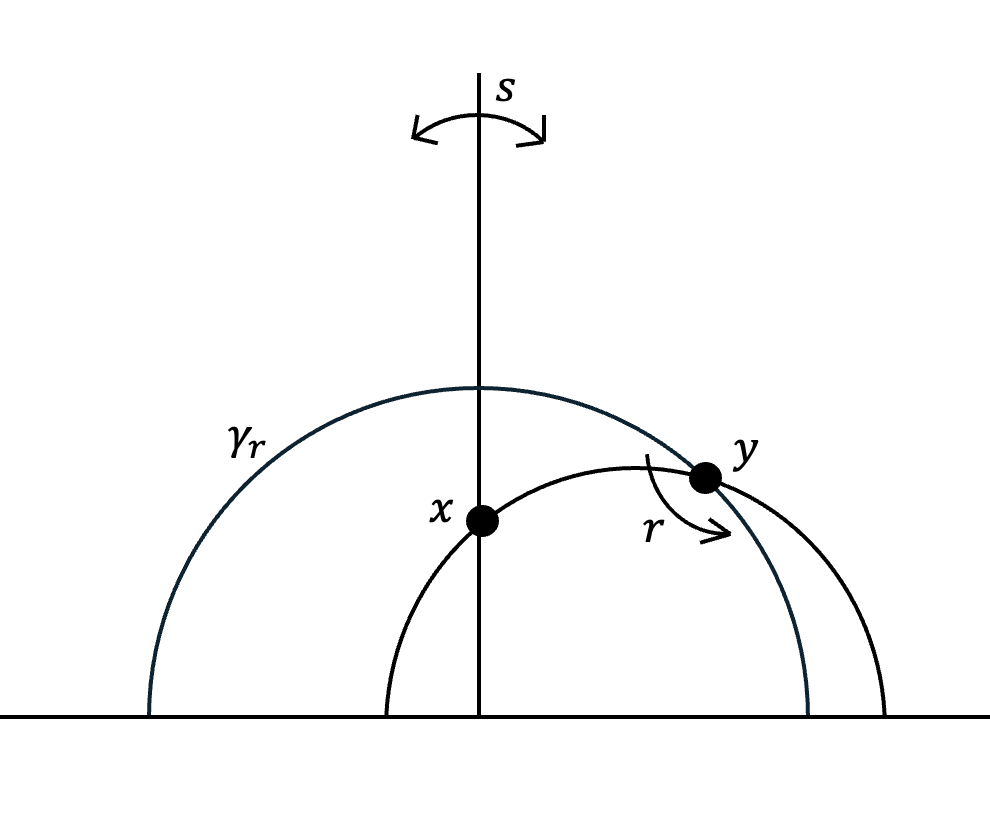}
 % \caption*{TestPicture}
%  \label{fig:test}
\end{figure}

Order of $ss_2$: As pointed out above, $s_2$ is the reflexion by the geodesic $s_1(\gamma_0)$, which is the half cercle perpendicular to the boundary intersecting the real line in the points $\pm 1+2 \mathrm{Re}(y)$. Since we are in the situation where $1/2\leq   \mathrm{Re}(y)$, we have 
$$\mathrm{Re}(x)\leq \frac{1}{2}\leq -1+2\mathrm{Re}(y).$$
In particular the two geodesics $\gamma_s$ and $s_1(\gamma_{s_0})$ intersect at most on the boundary and thus the corresponding group generated by their reflexions is again the infinite dihedral group. 
%we can conclude that the reflecting geodesic $s_1(\gamma_0)$ does not intersect $\gamma_s$, neither in the hyperbolic plane, nor on its boundary (the latter is vertical over $\mathrm{Re}(x)\leq 0$ and $0<\pm 1+2 \mathrm{Re}(y)$). It folows that there exists a (unique) geodesic $\gamma$  perpendicular to both $\gamma_s$ and $s_1(\gamma_0)$. The actions of $s$ and $s_2$ restrict to an action by reflection on $\gamma$. Since the fixed points of $s$ and $s_2$ in $\gamma$ are distinct we conclude again that $\langle s,s_2\rangle=\langle s\rangle *\langle s_2\rangle \cong D_\infty$ and the order of $ss_2$ is infinite, which finishes the proof of the proposition. 
%$$PICTURE \ 3 \ Prop$$
\end{proof}

\section{Comparison of the growth of $S$ and $T= S\backslash \{ s_1\}\cup \{s_2,r\}$}\label{Section Comparison of growth}

In this section we prove that blow downs along pseudo-transpositions cannot increase the exponential growth rate. Let us start by giving a precise version of Theorem \ref{thm main blow down}:

\begin{thm} \label{replacement decrease} Let $(W,S)$ be a Coxeter system. Suppose that $s_1\in S$ is a pseudo-transposition and let $s_0\in S$ be the unique element such that $m_S(s_0,s_1)=2(2k+1)$ for $k\geq 1$. Set
$$ s_2:=s_1s_0s_1, \quad r:=(s_1s_0)^{2k+1}$$
and
$$T:= S\backslash \{ s_1\}\cup \{s_2,r\}.$$
Then
$$\omega(W,T)\geq \omega(W,S).$$
\end{thm}

\begin{proof} We apply the following strategy: Let $0<R_S, R_T\leq 1$ denote the radius of convergence of $W_S$ and $W_T$ respectively and set $R=\mathrm{min} \{R_S,R_T\}$. In virtue of Pringsheim's Theorem the growth function $W_T$ diverges when evaluated on its radius of convergence $R_T$ (and obviously converges for any real value $0\leq x<R_T$).   It is thus enough to show that
$$W_S(x)< W_T(x)$$
for $0<x<R$ which immediately implies that $R_S\geq R_T$ and hence $\omega(W,S)=1/R_S \leq 1/R_T = \omega(W,T)$. Since $W_S(x)$ and $W_T(x)$ are strictly positive for $0<x<R$ this will be equivalent to showing
$$ \frac{1}{W_S(x)}-\frac{1}{W_T(x)}=\frac{W_T(x)-W_S(x)}{W_S(x)W_T(x)}> 0.$$

Motivated by the observation that our generating sets $S$ and $T$ as in the theorem share a lot of subsets generating a finite group, we want to use Steinberg's Formula (Theorem \ref{Steinberg}). Define
$$S_2:=\{ s\in S\mid m(s,s_1)=2\} \ \mathrm{and} \ S_\infty:=\{ s\in S\mid m(s,s_1)=\infty\}.$$
By construction we have 
$$S=\{s_0,s_1\} \amalg S_2 \amalg S_\infty \ \mathrm{and} \ T=\{s_0,s_2,r\}  \amalg S_2 \amalg S_\infty.$$
Exploiting Proposition \ref{prop labels}, we summarize the labels appearing between $S_2\amalg S_\infty$ and the remaining two respectively three elements in $S$ and $T$  in the following diagram:

\begin{figure}[H]%                 use [hb] only if necceccary!
  \centering
  \includegraphics[width=10cm]{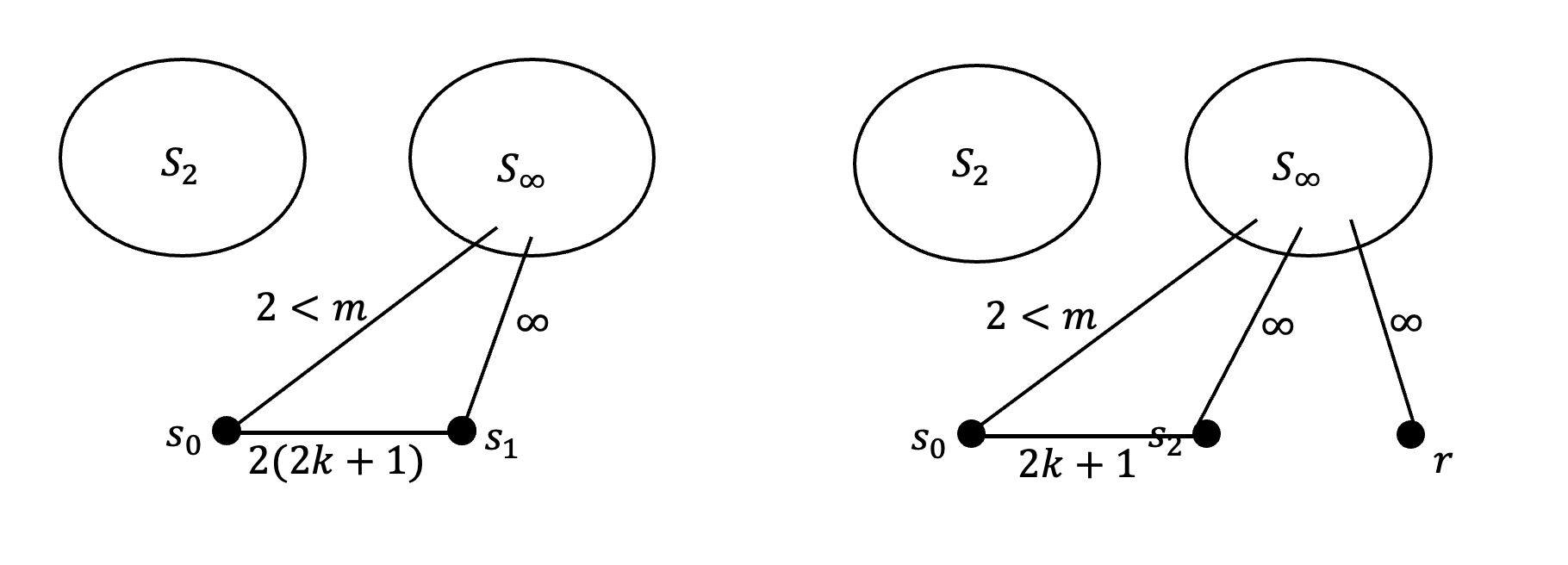}
 % \caption*{TestPicture}
%  \label{fig:test}
\end{figure}
\noindent Note that the labels between elements of $S_2 \amalg S_\infty$ are  omitted and are identical in both Coxeter diagrams. 

Observe that there is a bijection between 
$$\{ U \in \mathcal{S}^S_{\mathrm{fin}}\mid U\cap S_\infty \neq \emptyset\} \ \mathrm{and} \ \{ U \in \mathcal{S}^T_{\mathrm{fin}}\mid U\cap S_\infty \neq \emptyset\} .$$
Indeed, a subset intersecting $S_\infty$ can only generate a finite subgroup if it is contained in $\{s_0\}\amalg S_2\amalg S_\infty$. Indeed, if in $S$ it would contain $s_1$, there would be $u\in U\cap S_\infty$ with $m(s_1,u)=\infty$ so that $\langle U\rangle$ would contain the infinite dihedral group. Likewise in $T$ since $m(r,u)=m(s_2,u)=\infty$ for every $u\in S_\infty$ by Proposition  \ref{prop labels}. Writing the difference
$$\frac{1}{W_S(x^{-1})}-\frac{1}{W_T(x^{-1})}$$
using Steinberg's formula (Theorem \ref{Steinberg}), we can thus restrict to taking the sum over subsets generating finite groups not intersecting $S_\infty$, so that they are contained in $\{s_0,s_1\}\amalg S_2$ for $S$ and $\{s_0,s_2,r\}\amalg S_2$ for $T$:
\begin{align*}
\frac{1}{W_S(x^{-1})}-\frac{1}{W_T(x^{-1})}&=&\sum_{U\subset \mathcal{S}^{\{s_0,s_1\}\amalg S_2}_{\mathrm{fin}}} \frac{(-1)^{|T|}}{W_U(x)}-\sum_{U\subset \mathcal{S}^{\{s_0,s_2,r\}\amalg S_2}_{\mathrm{fin}}} \frac{(-1)^{|T|}}{W_U(x)}\\\\
&=&\frac{1}{W_{\{s_0,s_1\}\amalg S_2}(x^{-1})}- \frac{1}{W_{\{s_0,s_2,r\}\amalg S_2}(x^{-1})},
\end{align*}
where for the last equality we applied Steinberg's formula to the groups $\langle \{s_0,s_1\}\amalg S_2 \rangle $ and  $\langle \{s_0,s_2,r\}\amalg S_2 \rangle $. Since elements of $S_2$ commute with $s_0,s_1,s_2$ and $r$, and further that $r$ commutes with $s_0$ and $s_2$, we can use the formula for products of growth functions \cite[VI.A, Proposition 4 (i)]{dlH} to obtain, after substituting $x^{-1}$ by $x$,
\begin{equation}\label{equation}\frac{1}{W_S(x)}-\frac{1}{W_T(x)} =\frac{1}{W_{S_2}(x)}\left( \frac{1}{W_{\{s_0,s_1\}}(x)}  -\frac{1}{W_{\{s_0,s_2\}}(z)W_{\{r\}}(x)} \right). \end{equation}
Since $\langle S_2\rangle$ is a subgroup of $W$ the radius of convergence of $W_{S_2}$ is greater or equal to $R$. In particular $0<W_{S_2}(x)<\infty$ for $x\in (0,R)$. Thus, the left hand side in (\ref{equation}) is greater or equal to $0$ if and only if the term inside the parenthesis on the right hand side is greater or equal to $0$. Knowing that the growth function with respect to the standard generators of a dihedral group of order $2\ell$ is 
$$1+2x+\ldots +2x^{\ell-1}+x^\ell=(1+x)(1+x+\ldots +x^{\ell-1})$$
we obtain
\begin{eqnarray*}
 \frac{1}{W_{\{s_0,s_1\}}(x)}  &-&\frac{1}{W_{\{s_0,s_2\}}(x)W_{\{r\}}(x)} \\
 &=& \frac{1}{(1+x)(1+x+\ldots +x^{4k+1})}-\frac{1}{(1+x)^2(1+x+\ldots+x^{2k})}\\
 &=&\frac{(1+x)(1+x+\ldots+x^{2k})-(1+x+\ldots +x^{4k+1})}{(1+x)^2 (1+x+\ldots +x^{4k+1})(1+x+\ldots+x^{2k})}.
\end{eqnarray*}

For $x\in (0,R)$ the denominator of this expresion is strictly positive so we only need to show that the numerator is positive for $x\in (0,R)$. The numerator equates to
$$x+x^2+\ldots+x^{2k}-x^{2k+2}-\ldots -x^{4k+1}=x(1-x^{2k+1})(1+x+\ldots+x^{2k-1}).$$
Since $0<x<R\leq 1$, this expression is indeed strictly positive, which finishes the proof of the theorem. 

\end{proof}

\section{Examples}
We present here a few examples illustrating our main result. The first family of examples shows that the inequality in Theorem \ref{replacement decrease} is strict for certain hyperbolic triangular groups. The second shows that the inequality can be an equality. Finally we show that the inequality 
$$\omega_\mathrm{Cox}(G)\geq \omega(G)$$
is strict for $G=C_2*\mathrm{Sym}(5)$.

\subsection{Growth rates of triangular groups}

Let $k,m\in \mathbb{N}$ with $n\geq 2$. Set $S=\{s_0,s_1,t\}$ and consider the triangular group $(W,S)$ given by the triangle below. The element $s_1$ is a pseudo-transposition, so that applying the procedure of Section \ref{section pseudo transposition}, the group also admits the Coxeter system $T=\{s_0,s_2,r,t\}$. 

\begin{figure}[H]%                 use [hb] only if necceccary!
  \centering
  \includegraphics[width=9cm]{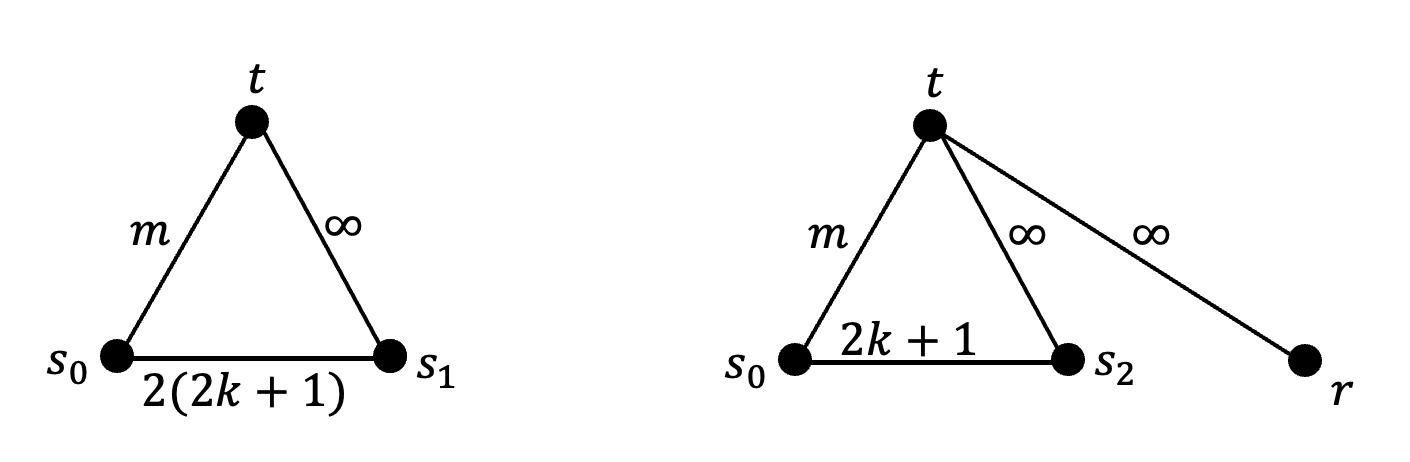}
 % \caption*{TestPicture}
%  \label{fig:test}
\end{figure}

We apply Steinberg's formula (Theorem \ref{Steinberg}) to obtain
$$W_S(z)=\frac{1+z-z^m-z^{m+1}-z^{4k+2}-z^{4k+3}+z^{4k+m+2}+z^{4k+m+3}}{1-2z+z^{m+1}+z^{4k+3}-z^{4k+m+2}}$$
and 
\[
\resizebox{1\linewidth}{!}{$
W_T(z)=\frac{1+2z+z^2-z^m-z^{m+1}-2z^{m+2}-z^{2k+1}-2z^{2k+2}-z^{2k+3}+z^{2k+m+1}+2z^{2k+m+2}+z^{2k+m+3}}{1-2z-z^2+2z^{m+1}+z^{2k+2}+z^{2k+3}-z^{2k+m+1}-z^{2k+m+2}}$.}
\]
Let $R_S$ and $R_T$ be the radii of convergence of $W_S$ and $W_T$ respectively. We will show that the inequality
$$\omega(W,S)<\omega(W,T)$$
is strict by showing that $R_S>R_T$. Again invoking  Pringsheim's Theorem, when the growth functions are rational, we know that the radius of convergence  equals  the smallest positive zero of the denominator. The claim will thus follow from the stronger strict inequality 
$$q_S(t)>q_T(t)$$ 
for any $t$ in the open interval $(0,1)$, where $q_S$ and $q_T$ are the denominators of $W_S$ and $W_T$ respectively. This shows the strict inequality since $W$ is hyperbolic and thus has exponential growth rate, so that $\omega(W,S)$ and $\omega(W,T)$ are strictly greater than $1$. The difference of $q_S$ and $q_T$ takes the form
\begin{align*}
q_S(t)-q_T(t)=&\ t^2-t^{m+1}-t^{2k+2}-t^{2k+3}+t^{2k+m+1}+t^{2k+m+2}+t^{4k+3}-t^{4k+m+2}\\
=&\ t^2(1-t^{m-1})(1+t^{4k+1}-t^{2k}-t^{2k+1})\\
=&\ t^2(1-t^{m-1})(1-t^{2k})(1-t^{2k+1})\\
>&\ 0
\end{align*}
for $t\in (0,1)$ since each factor is $>0$. 

As an example, for $m=2$ and $k=1$ we obtain
$$\omega(W,S)\sim 1.57015... < 1.83928...\sim \omega(W,T).$$ 

\subsection{Examples with equal growth rate} It is easy to somehow artificially produce examples of Coxeter groups with pseudo-transpositions for which the growth rates with respect to $S$ or $T$ (as in Theorem  \ref{replacement decrease}) are equal. Let $(W_1,S_1)$ be a Coxeter system with a pseudo-transposition $s_1$. (For example $W_1$ could be the dihedral group of order $2(2k+1)$ for some $k\geq 1$ with its canonical Coxeter generators. Or a hyperbolic triangular group as in the previous example.) Let $T_1$ be the Coxeter generating set of $W_1$ as in Theorem \ref{replacement decrease}. Let $(W_2,S_2)$ be another Coxeter system such that $\omega(W_2,S_2)\geq \omega(W_1,T_1)$. (For example $W_2$ could be a free product of sufficiently many copies of $C_2$ with the canonical generators. Or in the case of $W_1$ dihedral it could be the trivial group.) We consider the product $W=S_1\times S_2$ generated by $S=S_1\cup S_2$. Then $s_1=(s_1,e)$ is also a pseudo-transposition in the product and the corresponding alternative generating set as in Theorem \ref{replacement decrease} is precisely $T=T_1\cup S_2$. But for products generated by a union of generating sets in each factor we know that \cite[VI.A, Proposition 4 (i)]{dlH}
$$\omega(W_1\times W_2, S_1\cup S_2)=\mathrm{max}(\omega(W_1, S_1),\omega( W_2,  S_2)).$$
Since by assumption and Theorem \ref{replacement decrease}
$$\omega(W_2,S_2)\geq \omega(W_1,T_1)\geq \omega(W_1,S_1),$$
it is immediate that 
$$\omega(W,S)=\omega(W_1\times W_2, S_1\cup S_2)=\omega( W_2,  S_2)=\omega(W_1\times W_2, T_1\cup S_2)=\omega(W,T).$$

\subsection{Non Coxeter generating sets}\label{Sym5} The aim of this example is to show that 
$$\omega_\mathrm{Cox}(C_2*\mathrm{Sym}(5))>\omega(C_2*\mathrm{Sym}(5)).$$
Let $t$ be the generator of $C_2$. In the symmetric group, we denote the elementary transpositions by $s_i=(i,\ i+1)$ for $1\leq i\leq n-1$ and set $\sigma=(1,\ 2, \ 3, \ 4, \ 5)$. In $\mathrm{Sym}(5)$ we consider the two standard generating sets
$$S_0=\{s_1,\ldots, s_{4}\} \quad \mathrm{and} \quad T_0=\{s_1,\sigma\},$$
giving the corresponding generating sets 
$$S=\{t\}\cup S_0 \quad \mathrm{and}\quad T=\{t\}\cup T_0$$
of the free product $C_2*\mathrm{Sym}(5)$. Clearly only $S$ is a Coxeter system. We start by showing that 
$$\omega(C_2*\mathrm{Sym}(5),S)>\omega(C_2*\mathrm{Sym}(5),T).$$
We first compute the two growth series in $\mathrm{Sym}(5)$. By Steinberg's formula (Theorem \ref{Steinberg}) we have
$$W_{S_0}(z)=1+4z+9z^2+15z^3+20z^4+22z^5+20z^6+15z^7+9z^8+4z^9+z^{10}.$$
For the growth serie with respect to $T_0$ we compute by brute force 
$$W_{T_0}(z)=1+3z+6z^2+10z^3+16z^4+24z^5+29z^6+21z^7 +6z^8+3z^9+z^{10}.$$
It is very surprising that no expression for the growth serie for these standard generating sets of $\mathrm{Sym}(n)$ given by one transposition and one $n$-cycle seems to be known in general. Not even the length of a longest element nor the uniqueness of the longest element is known \cite{SymAI}.

Applying the formula for growth of free products  \cite[VI.A, Proposition 4 (ii)]{dlH} we obtain
$$W_S(z)=\frac{W_{S_0}(z)(1+z)}{1-4z^2 -9z^3 -15z^4 -20z^5 -22z^6 -20z^7 -15z^8 -9z^9 -4z^{10}-z^{11}}$$
 and 
 $$W_T(z)=\frac{W_{T_0}(z)(1+z)}{1-3z^2 -6z^3 -10z^4 -16z^5 -24z^6 -29z^7 -21z^8 -6z^9 -3z^{10} -z^{11}}.$$ 
 
 The radius of convergence of $W_S$ is $0.324227...$ and the one of $W_T$ is the slightly greater $0.358345...$. In particular we indeed obtain the strict inequality
 $$\omega(W,S)>\omega(W,T).$$
 Observe that for the corresponding growth rates for $n=3$ we obtain the inverse inequality, and for $n=4$ there is equality due to the fact that the Cayley graphs of the two generating sets on $\mathrm{Sym}(4)$ are by accident isomorphic. 
 
 Finally we show that the Coxeter group $C_2*\mathrm{Sym}(5)$ is rigid in the sense that any Coxeter system has to be conjugated to $S$. It could be that this is well known but we were unable to locate a proof in the litterature. In any case this allows to conclude that 
$$\omega_\mathrm{Cox}(C_2*\mathrm{Sym}(5))=\omega(C_2*\mathrm{Sym}(5),S)>\omega(W,T)\geq\omega(C_2*\mathrm{Sym}(5)).$$
 Let $S'$ be a Coxeter system for $C_2*\mathrm{Sym}(5)$. We consider the action of  $C_2*\mathrm{Sym}(5)$ on its associated Bass-Serre tree. Elements of $C_2*\mathrm{Sym}(5)$  have finite order if and only if they are stabilizers of vertices and hence conjugated to either $C_2$ or $\mathrm{Sym}(5)$. The product of two stabilizers of distinct vertices is a hyperbolic transformation of the tree and hence has infinite order. It follows that if we omit the $\infty$ labels (but keep all finite $\geq 2$ labels) in the Coxeter diagram of $S'$ we obtain as many connected components as the set $\{\mathrm{Stab}(s)\mid s\in S'\}$, say $c$. But the abelianization of such a group is $C_2^c$. Since the abelianization of $C_2*\mathrm{Sym}(5)$ is $C_2^2$ this forces $c=2$. In order to generate the full group these two connected components have to stabilize adjacent edges so that we can up to conjugating suppose that they lie either in $C_2$ or $\mathrm{Sym}(5)$. Now by the classification of finite Coxeter systems (and the fact that $\mathrm{Sym}(5)$ does not admit a non trivial factorization) we conclude that $S'$ is conjugated to $S$. 
 
\section{Proof of Theorem \ref{thm even}} \label{algo}

We shortly recall Mihalik's algorithm which for an even Coxeter group $W$, starting from any Coxeter system containing an odd label, constructs a Coxeter system with one less generator, and one less odd label. Iterating this procedure leads to the unique even Coxeter system for $W$. This is done in two steps: Diagram twisting and blow down along a pseudo-transposition. In presence of an odd label $m(s_0,s_2)=2k+1$, one would like to perform a blow down along some dihedral group of order $4(2k+1)$ containing $\langle s_0,s_2\rangle$ as an index $2$ subgroup. This is possible if (and only if) the Coxeter diagram is as pictured on the right hand side of the illustration in Section \ref{Section Comparison of growth}. The purpose of the diagram twisting is to modify the initial Coxeter system to be in a situation where the blow down can be applied.

%then replacing a triangle by an edge. The latter procedure is precisely the inverse of exploiting pseudo-transpositions to produce a Coxeter system with one more generator. In view of the shape of the Coxeter group on the right hand side of the Illustration REF \marginpar{ref \`a la figure 3.1} if there is an odd edge between vertices $s_0,s_2$, say $m(s_0,s_2)=2k+1$ such that one of the vertex, say $s_2$, has the property that $m(s,s_2)\in \{2,\infty\}$ for any $s\in S\setminus \{s_0,s_2\}$. This is the purpose of the diagram twisting. 

\subsection*{First step: Diagram twisting} We begin with the standard construction of diagram twisting from \cite{BradyEtal}. Before that recall that in a finite Coxeter group $W_V$, there is a unique longest element, $\Delta$ and its action by conjugation induces  a permutation of its Coxeter system $V$. 

\begin{dfn} \label{diag twist} Let $\Gamma$ be a Coxeter diagram with vertex set $S$. Let $U$ and $V$ be disjoint subset of $S$ such that
\begin{enumerate}
\item $W_V$ is a finite Coxeter group,
\item for every $s\in S\setminus (U\cup V)$ such that there exists $u\in U$ with $m(s,u)<\infty$ it holds that $m(s,v)=2$ for every $v\in V$.
\end{enumerate}
The diagram twisting of $\Gamma$ obtained by twisting $U$ by $\Delta$, where $\Delta$ is the longest element in $W_V$, is the diagram with vertex set $S$ and the same labels as in $\Gamma$ except between vertices $u\in U$ and $v\in V$ where one sets
$$m'(u,v):=m(u,\Delta v\Delta^{-1}).$$
\end{dfn}

It is not difficult to prove that the resulting Coxeter system $S'$ generates a group isomorphic to the initial Coxeter group \cite[Theorem 4.5]{BradyEtal}. Furthermore, since there is a bijection between the finite parabolic subgroups of $S$ and $S'$, Steinberg's formula (Theorem \ref{Steinberg}) implies that the two Coxeter systems have the same growth serie and in particular the same exponentional growth rate. 

In our situation, we will take $V=\{s_0,s_2\}$ for $s_0,s_2$ with $m(s_0,s_2)=2k+1$. Conjugation by the longest element in the dihedral group $W_V$ of order $2(2k+1)$ will simply exchange $s_0$ and $s_2$. In what follows we will adopt the convenient convention to omit edges labelled by $\infty$ from the Coxeter diagram $\Gamma$ (but keep all $2$ labels). Define
$$\mathrm{Lk}_2(s_0,s_2):=\{s\in S\setminus \{s_0,s_2\}\mid m(s,s_0)=m(s,s_2)=2\}$$
and take for $U$ the union of the connected components of $S\setminus (\{s_0,s_2\}\cup \mathrm{Lk}_2(s_0,s_2))$ which are connected to $s_2$ by at least one edge. The pair $(U,\{s_0,s_2\})$ satisfies the condition for diagram twisting of Definition \ref{diag twist}. Indeed, let $s\in S\setminus (U\amalg \{s_0,s_2\})$ and $u\in U$ such that $m(s,u)<\infty$. If $s\in \mathrm{Lk}_2(s_0,s_2)$ we already know that $m(s,s_0)=m(s,s_2)=2$. Suppose $s\notin  \mathrm{Lk}_2(s_0,s_2)$ then $s\in S\setminus \{s_0,s_2\}\cup \mathrm{Lk}_2(s_0,s_2)$. If $m(s,s_2)<\infty$ then $s$ belongs to $U$ by construction, which is excluded. Finally, if $m(s,s_0)<\infty$, then we obtain a path from $s_0$ to $s_2$ by first taking the edge from $s_0$ to $s_2$, then the edge from $s_2$ to $u$, and then a path connecting $u$ to $s_2$ which exists by the construction of $U$. But his contradicts \cite[Lemma 4]{Mihalik}, which claims that any path from $s_0$ to $s_2$ either intersects $\mathrm{Lk}_2(s_0,s_2) $ or contains the edge from $s_0$ to $s_2$, neither of which hold here. 

Finally it remains to establish that the diagram $\Gamma'$ obtained by twising $U$ along $\Delta$ has the desired property that $m(s,s_2)\in \{2,\infty\}$ for any $s\in S\setminus \{s_0,s_2\}$. There is nothing to prove for $s\in \mathrm{Lk}_2(s_0,s_2)$. If $s\in U$ then the label $m'(s,s_2)$ is the label $m(s,s_0)$, which cannot be finite otherwise this would create a path from $s_0$ to $s\in U$ to $s_2$ contradicting \cite[Lemma 4]{Mihalik} once again. Finally if $s$ is neither in $ \mathrm{Lk}_2(s_0,s_2)$ nor $U$ then a finite label $m(s,s_2)$ would contradict the definition of $U$. 

Diagram twisting does not affect the exponential growth rate. 

\subsection*{Second step: Pseudo-transposition} With the diagram twisting we know that our Coxeter diagram now takes the form 

\begin{figure}[H]%                 use [hb] only if necceccary!
  \centering
  \includegraphics[width=6cm]{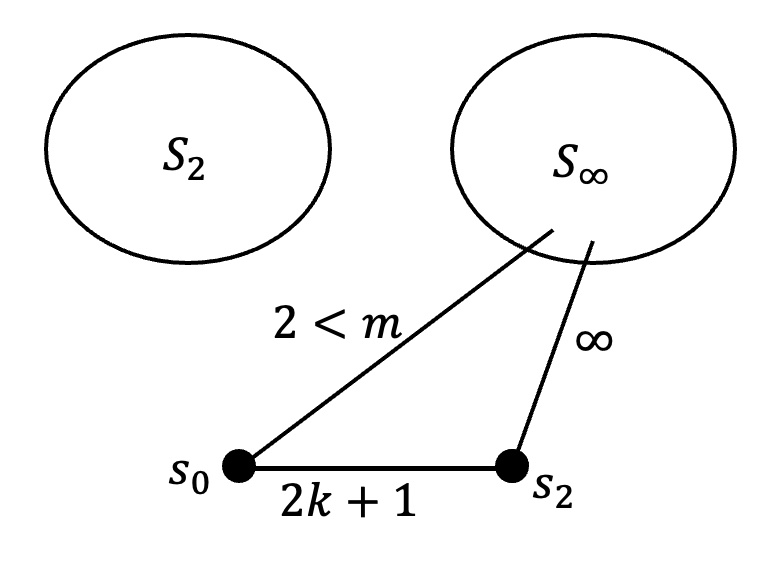}
 % \caption*{TestPicture}
%  \label{fig:test}
\end{figure}
\noindent where $S_2=\{s\in S\setminus\{s_0,s_2\}\mid m(s,s_2)=2\}$ and $S_\infty=\{s\in S\setminus\{s_0,s_2\}\mid m(s,s_2)=\infty\}$. We can now invoke  \cite[Proposition 5]{Mihalik} which allows us  to conclude that there exists $r\in S$ such that 
\begin{enumerate}
\item $r\in \mathrm{Lk}_2(s_0,s_2)$ or equivalently $m(r,s_0)=m(r,s_2)=2$,
\item if $m(r,s)<\infty$ for $s\in S\setminus \{s_0,s_2\}$ then $m(r,s)=2$ and $s\in S_2$,
\item any maximal parabolic containing $x$ and $y$ contains $r$. 
\end{enumerate}
In the last item a maximal parabolic is a finite subgroup of $W$ generated by a subset of $S$. This conclusion is stated more generally in \cite[Proposition 5]{Mihalik} for arbitrary parabolics (called there simplices) under an additional condition of being conjugated to a parabolic of the unique even Coxeter diagram. This condition is void for maximal parabolics. The second condition implies that $m(r,s)=\infty$ for every $s\in S_\infty$. It remains to show that $m(r,s)=2$ for every $s\in S_2$: If $s\in S_2$, then $x,y,s$ belong to a unique maximal parabolic, which by item 3) must contain $r$. In particular the subgroup generated by $s$ and $r$ is finite, so that $m(r,s)<\infty$ and hence $m(r,s)=2$ by item 2). It follows that our diagram precisely has the shape described on the right hand side of the illustration in Section \ref{Section Comparison of growth} and thus can be obtained by the pseudo-transposition procedure described in Section \ref{Section Comparison of growth}. 

As proven in Theorem \ref{replacement decrease}, this procedure decreases the exponential growth rate. 

\bibliographystyle{alpha}

\end{document}